\documentclass[letterpaper]{article}
\usepackage{iclr2020_conference,times}

\usepackage{amsmath,amsfonts,bm}

\def\eqref#1{equation~\ref{#1}}
\def\1{\bm{1}}

\DeclareMathAlphabet{\mathsfit}{\encodingdefault}{\sfdefault}{m}{sl}
\SetMathAlphabet{\mathsfit}{bold}{\encodingdefault}{\sfdefault}{bx}{n}

\usepackage[utf8]{inputenc} % allow utf-8 input
\usepackage[T1]{fontenc}    % use 8-bit T1 fonts
\usepackage{hyperref}       % hyperlinks
\usepackage{url}            % simple URL typesetting
\usepackage{booktabs}       % professional-quality tables
\usepackage{amsfonts}       % blackboard math symbols
\usepackage{amsmath}
\usepackage{nicefrac}       % compact symbols for 1/2, etc.
\usepackage{microtype}      % microtypography
\usepackage{graphicx}
\usepackage{subcaption}
\usepackage{wrapfig}
\usepackage{algorithmic}
\usepackage{enumitem}
\usepackage{amsthm}

\usepackage[utf8]{inputenc} % allow utf-8 input
\usepackage[T1]{fontenc}    % use 8-bit T1 fonts
\usepackage{hyperref}       % hyperlinks
\usepackage{url}            % simple URL typesetting
\usepackage{booktabs}       % professional-quality tables
\usepackage{amsfonts}       % blackboard math symbols
\usepackage{nicefrac}       % compact symbols for 1/2, etc.
\usepackage{microtype}      % microtypography
\usepackage{amsmath}
\usepackage{xcolor}
\usepackage{amsthm}
\usepackage{wrapfig}
\usepackage[ruled,vlined]{algorithm2e}
\usepackage{algorithmic}
\usepackage{comment}
\usepackage{pdfpages}
\usepackage{amssymb}

\SetKwInput{KwInput}{In}
\SetKwInput{KwOutput}{Out}

\newtheorem{proposition}{Proposition}

\def\d{\text{d}}
\def\vec{\mathrm{vec}}
\def\v{\text{vec}}

\title{Infinite-Horizon Differentiable Model \\ Predictive Control}

\author{
  Sebastian East\textsuperscript{1,2}\thanks{Corresponding author.} , Marco Gallieri\textsuperscript{1}, Jonathan Masci\textsuperscript{1}, Jan Koutn\'{i}k\textsuperscript{1} \& Mark Cannon\textsuperscript{2} \\
  \textsuperscript{1}NNAISENSE, Lugano, Switzerland \\
  \textsuperscript{2}Department of Engineering Science, University of Oxford, Oxford, UK \\
  \texttt{\{sebastian.east,mark.cannon\}@eng.ox.ac.uk} \\
  \texttt{\{marco,jonathan,jan\}@nnaisense.com} \\
}

\iclrfinalcopy

\begin{document}

\maketitle

\begin{abstract}
This paper proposes a differentiable linear quadratic Model Predictive Control (MPC) framework for safe imitation learning. The infinite-horizon cost is enforced using a terminal cost function obtained from the discrete-time algebraic Riccati equation (DARE), so that the learned controller can be proven to be stabilizing in closed-loop. A central contribution is the derivation of the analytical derivative of the solution of the DARE, thereby allowing the use of differentiation-based learning methods. A further contribution is the structure of the MPC optimization problem: an augmented Lagrangian method ensures that the MPC optimization is feasible throughout training whilst enforcing hard constraints on state and input, and a pre-stabilizing controller ensures that the MPC solution and derivatives are accurate at each iteration. The learning capabilities of the framework are demonstrated in a set of numerical studies. 
\end{abstract}

\section{Introduction}
Imitation Learning (IL, \citealp{osa_algorithmic_2018}) aims at reproducing an existing control policy by means of a function approximator and can be used, for instance, to hot-start reinforcement learning. Effective learning and generalisation to unseen data are paramount to IL success, especially in safety critical applications.  Model Predictive Control (MPC, \citealp{Maciejowski_book, Camacho2007, rawlingsMPC,  Cannon_book, Gallieri2016,  Borrelli_book, Rakovic2019}) is the most successful advanced control methodology for systems with \emph{hard safety constraints}. At each time step, a finite horizon forecast is made from a predictive model of the system and the optimal actions are computed, generally relying on convex constrained Quadratic Programming (QP, \citealp{Boyd_convexopt, Bemporad2000}). Stability of the MPC in closed loop with the physical system requires the solution of a simpler unconstrained infinite horizon control problem \citep{rawlings_mayne_paper} which results in a value function (terminal cost and constraint) and a candidate terminal controller to be accounted for in the MPC forecasting. For Linear Time Invariant (LTI) models and quadratic costs, this means solving (offline) a Riccati equation \citep{KalmanLQR} or a linear matrix inequality \citep{Boyd_lmi}. Under these conditions, an MPC controller will effectively control a system, up to a certain accuracy, provided that uncertainties in the model dynamics are limited \citep{Limon2009}. Inaccuracies in the MPC predictions can reduce its effectiveness (and robustness) as the forecast diverges from the physical system trajectory over long horizons. This is particularly critical in applications with both short and long-term dynamics and it is generally addressed, for instance in robust MPC \citep{richards_a._g._robust_2004, Rakovic2012}, by using a controller to pre-stabilise the predictions.

This paper presents an infinite-horizon differentiable linear quadratic MPC that can be learned using gradient-based methods. In particular, the learning method uses an MPC controller where the terminal cost and terminal policy are the solution of an unconstrained infinite-horizon Linear Quadratic Regulator (LQR). A closed-form solution for the derivative of the Discrete-time Algebraic Riccati Equation (DARE) associated with the LQR is presented so that the stationary solution of the forward pass is fully differentiable. This method allows analytical results from control theory to be used to determine the stabilizing properties of the learned controller when implemented in closed-loop. Once the unconstrained LQR is computed, the predictive model is pre-stabilised using a linear state-feedback controller to improve the conditioning of the QP and the numerical accuracy of the MPC solution and gradients. The proposed algorithm successfully learns an MPC with both local stability and intrinsic robustness guarantees under small model uncertainties.  

\paragraph{Contributions}
This paper provides a framework for correctly learning an infinite-horizon, LTI quadratic MPC using recent developments in differentiable QPs \citep{amos_optnet:_2017}  and principles from optimal control \citep{Blanchini}. A~primary contribution is that the Discrete-time Algebraic Riccati Equation (DARE) is used to provide infinite-horizon optimality and stability, and an analytical derivative of the solution of the DARE is derived so that differentiation-based optimization can be used for training.  This connects known results on MPC stability \citep{Limon2003StableCM,Limon2009} and on infinite-horizon optimality \citep{Scokaert1998} to imitation learning \citep{osa_algorithmic_2018}.

A further contribution is the MPC control formulation: a pre-stabilizing linear state-feedback controller is implemented from the solution of the DARE, and then the total control input is obtained as a perturbation of the feedback control law from the solution of a convex QP. The pre-stabilizing controller ensures that the QP is well conditioned and promotes a highly accurate global solution, which in turn ensures that the gradients calculated in the backwards pass are accurate. Additionally, an augmented Lagrangian penalty method is used to enforce constraints on state and control input. This approach ensures that the hard constraints are strictly enforced if the penalty term is sufficiently large, and also guarantees that the MPC problem is feasible throughout the training process. These contributions are in contrast to \citep{amos_differentiable_2018} which did not consider state constraints, and implemented a differential dynamic programming \citep{tassa_control-limited_2014} method to solve the MPC optimization for which convergence could not be guaranteed.

The framework is implemented on a set of second order mass-spring-damper systems and a vehicle platooning model, where it is demonstrated that the infinite horizon cost can be learned and the hard constraints can be guaranteed using a short finite prediction horizon.

\paragraph{Notation}
$\mathbf{I}_n:=$ $n\times n$ identity matrix. $\mathbf{O}_{m \times n}:=$ $m \times n$ matrix of zeros. $\mathbf{0}_n:=$ a vector of $n$ zeros. $\mathbf{1}_n:=$ a vector of $n$ ones. All inequalities $\leq$ and $\geq$ are considered element-wise in the context of vectors. $\rho (A):=$ largest absolute eigenvalue of given matrix $A$. $\text{vec}: \mathbb{R}^{m \times n} \mapsto \mathbb{R}^{mn}$ is defined as $\text{vec}\left( [c_1  \cdots  c_n]  \right) := ( c_1, \cdots, c_n),$ i.e. the columns of a matrix stacked into a vector. For a matrix $A \in \mathbb{R}^{m \times n}$, the $\mathbf{V}_{m,n} \in \mathbb{R}^{mn \times mn}$ permutation matrix is implicitly defined by $\mathbf{V}_{m,n} \v A := \v A^\top$. The Kronecker product, $\otimes$, is defined as in \cite[pp. 440]{magnus99}.

\section{Differentiable MPC}\label{section_diff_mpc}
\paragraph{Linear quadratic MPC}

This paper considers linear time invariant systems of the form
\begin{equation}\label{system_dynamics}
x_{t+d t} = A x_t + B u_t,
\end{equation}
where $x_t \in \mathbb{R}^n$ is the system state, $u_t \in \mathbb{R}^m$ is the control input, $A\in \mathbb{R}^{n \times n}$ is the state transition matrix, $B \in \mathbb{R}^{n \times m}$ is the input matrix, $t \in \mathbb{R}$ is the time, and $d t \in \mathbb{R}$ is the timestep (assumed constant). The control problem for this system is to determine the sequence of values of $u_t$ that achieve a desired level of performance (e.g. stability, frequency response, etc...), and when the system is subject to hard constraints on control input, $u_t\in\mathbb{U}$, and state, $x_t\in\mathbb{X}$, (or a combination of both), a well studied framework for controller synthesis is MPC. The principle of MPC is that the system's control input and state are optimized over a finite prediction horizon, then the first element of the obtained control sequence is implemented at the current time step and the process is repeated \textit{ad infinitum}. For linear MPC it is common to use a quadratic stage cost and box constraints on state and control ( $\underline{x} \leq x_k \leq \overline{x}$ and $\underline{u} \leq u_k \leq \overline{u}$ where $\underline{u} \leq 0 \leq \overline{u}$), so that at each time index $t$ the vector of optimized control variables $\hat{u}^\star$ is obtained from

\begin{equation}\label{MPC_1}
\begin{aligned}
\hat{u}^{\star}_{0:N} = \underset{\hat{u}}{\text{argmin}} \ & \frac{1}{2}\sum_{k=0}^{N-1} \hat{u}_k^\top R \hat{u}_k + \frac{1}{2} \sum_{k=1}^{N-1} \hat{x}_k^\top Q \hat{x}_k + \frac{1}{2} \hat{x}_N^\top Q_N \hat{x}_N + k_u \sum_{k=0}^{N-1} \textbf{1}_m^\top r_k + k_x \sum_{k=1}^{N} \textbf{1}_n^\top s_k  \\
\text{s.t.} \ & \hat{x}_0 = x_t, \\
& \hat{x}_{k+1} = A \hat{x}_k + B \hat{u}_k,  \quad k \in \{0,\dots,N-1\}, \\
& \underline{u} - r_k \leq \hat{u}_k \leq \overline{u} + r_k \quad \text{and} \quad r_k \geq 0, \quad k \in \{ 0, \dotsm N-1 \}, \\
& \underline{x} - s_k \leq \hat{x}_k \leq \overline{x} + s_k \quad \text{and} \quad  s_k \geq 0, \quad k \in \{1,\dots,N\},
\end{aligned}
\end{equation}

where $\hat{u}_{0:N}$ is the predicted control trajectory, $\hat{x}$ is the predicted state trajectory, $R \in \mathbb{R}^{m \times m} \succeq 0$ represents the stage cost on control input, $Q \in \mathbb{R}^{n \times n} \succeq 0$ represents the stage cost on state,  $Q_N \in \mathbb{R}^{n \times n} \succeq 0$ represents the terminal cost on state, $N\in \mathbb{N}$ is the prediction horizon, $r_k\in \mathbb{R}^m$ are slack variables for the control constraint, $s_k \in \mathbb{R}^n$ are slack variables for the state constraint, and $k_u \in \mathbb{R} > 0$ and $k_x \in \mathbb{R} > 0$ represent the cost of control and state constraint violations. The variables $s$ and $r$ enforce the box constraints on state and control using the augmented Lagrangian method \cite[\S 17.2]{nocedal2006a}, and it can be shown that for sufficiently high $k_x$ and $k_u$ the constraints $\underline{x} \leq x_k \leq \overline{x}$ and $\underline{u} \leq u_k \leq \overline{u}$ can be \textit{exactly guaranteed} \citep{Kerrigan00softconstraints} (i.e. $s = r = 0$). The benefit of this approach is that it ensures that the MPC optimization is feasible at each iteration of the learning process, whilst still ensuring that the constraints are `hard'. To close the MPC control loop, at each timestep, $t$, the first element of the optimized control sequence, $\hat{u}^\star_0$, is implemented as $u_t$.

\paragraph{Pre-stabilised MPC}

If the control input is decomposed into $u_t = K x_t + \delta u_t$, where $K \in \mathbb{R}^{m \times n}$ is a \emph{stabilizing} linear state-feedback matrix and $\delta u_t$ is a perturbation to the feedback control, system (\ref{system_dynamics}) becomes
\begin{equation}\label{stable_system_dynamics}
x_{t+d t} = (A + BK) x_t + B \delta  u_t,
\end{equation}
and problem (\ref{MPC_1}) becomes
\begin{equation}\label{MPC_2}
\begin{aligned}
\delta  \hat{u}^{\star}_{0:N} = \underset{\delta\hat{u}}{\text{argmin}} \ & \frac{1}{2}\sum_{k=0}^{N-1} (K \hat{x}_k + \delta  \hat{u}_k)^\top R (K \hat{x}_k + \delta  \hat{u}_k) + \frac{1}{2} \sum_{k=1}^{N-1} \hat{x}_k^\top Q \hat{x}_k + \frac{1}{2} \hat{x}_N^\top Q_N \hat{x}_N \\
& + k_u \sum_{k=0}^{N-1} \textbf{1}_m^\top r_k + k_x \sum_{k=1}^{N} \textbf{1}_n^\top s_k \\
\text{s.t.} \ & \hat{x}_0 = x_t, \\
& \hat{x}_{k+1} = (A + BK) \hat{x}_k + B \delta \hat{ u}_k,  \quad k \in \{0,\dots,N-1\}, \\
& \underline{u} - r_k \leq K \hat{x}_k + \delta  \hat{u}_k \leq \overline{u} + r_k \quad \text{and} \quad r_k \geq 0, \quad k \in \{0,\dots, N-1\}, \\
& \underline{x}-s_k \leq \hat{x}_k \leq \overline{x}+s_k \quad \text{and} \quad  s_k \geq 0, \quad k \in \{1, \dots, N \},
\end{aligned}
\end{equation}

so that $\hat{u}^\star_0=Kx_t + \delta  \hat{u}^\star_0$ is implemented as $u_t$. Using this decomposition, system (\ref{stable_system_dynamics}) controlled with the solution of (\ref{MPC_2}) is \textit{precisely equal} to system (\ref{system_dynamics}) controlled with the solution of (\ref{MPC_1}), but problem (\ref{MPC_2}) is preferable from a computational standpoint if $A$ is open-loop unstable (i.e. $\rho (A)>1$) and $N$ is `large', as this can lead to poor conditioning of the matrices defined in Appendix \ref{appendix_MPC_1}. This is important in the context of differentiable MPC, as if $A$ is being learned then there may be no bounds on its eigenvalues at any given iteration.

\paragraph{MPC derivative.}

Problems (\ref{MPC_1}) and (\ref{MPC_2}) can be rearranged into the QP form (details in Appendix \ref{appendix_MPC_1})
\begin{equation}\label{QP}
\begin{aligned}
z^{\star} = \underset{z}{\text{argmin}} \  \frac{1}{2} z^\top Hz + q^\top z \quad \text{s.t.} \ \  l_b \leq Mz \leq u_b.
\end{aligned}
\end{equation}
When $z^\star$ is uniquely defined by (\ref{QP}), it can also be considered as the solution of an implicit function defined by the Karush-Kuhn-Tucker (KKT) conditions, and in \cite{amos_optnet:_2017} it was demonstrated that it is possible to differentiate through this function to obtain the derivatives of $z^\star$ with respect to the parameters $H$, $q$, $l$, $M$, and $u$. \footnote{Note that ($\ref{QP}$) differs from the form presented in \cite{amos_optnet:_2017}, and is instead the form of problem solved by the OSQP solver used in this paper. Appendix \ref{appendix_OSQP} demonstrates how to differentiate (\ref{QP}) using the solution returned by OSQP.} The MPC controller can then be used as a layer in a neural network, and backpropagation can be used to determine the derivatives of an imitation cost function with respect to the MPC parameters $Q$, $R$, $A$, $B$, $\underline{u}$, $\overline{u}$, $\underline{x}$, $\overline{x}$, $k_x$ and $k_u$. 

\paragraph{Imitation Learning.} A possible use case of the derivative of a model predictive controller is imitation learning, where a subset of $\{$cost function, system dynamics, constraints$\}$ are learned from observations of a system being controlled by an `expert'. Imitation learning can be performed by minimizing the loss
\begin{equation}\label{eq:imitation_loss_mixed}
     \frac{1}{T}\sum_{t=0}^{T} {\|u_{t:t+Ndt}-\hat{u}^\star_{0:N}(x_t)\|_2^2 + \beta \|\hat{w}_t\|_2^2}, 
\end{equation}
where $u_t$ is the measured control input, $\hat{u}^\star_{0:N}(x_t)$ is the full MPC solution, and $\beta\geq0$ is a hyperparameter. It is assumed that both the learning algorithm and MPC controller have completely precise measurements of both the state and control input. The first term of (\ref{eq:imitation_loss_mixed}) is the control imitation loss, and the second term penalises the one-step ahead prediction error $\hat{w}_t=Ax_t+Bu_t - x_{t+dt}.$ In practice, the prediction error loss might not be needed for the MPC to be learned correctly, however its use can be instrumental for stability, as discussed in the next section. 
\section{Terminal cost for infinite horizon}

% {\color{red} short intro}

\paragraph{Terminal cost.}

The infinite-horizon discrete-time Linear Quadratic Regulator (LQR, \citealp{KalmanLQR}) is given with state feedback gain 
\begin{equation}\label{LQR}
K =  - (R + B^\top P B)^{-1} B^\top P A,
\end{equation}
where $P$ is obtained as a solution of the DARE
\begin{equation}\label{DARE}
P = {A}^\top P {A} - {A}^\top P B (R + B^\top P B)^{-1} B^\top P {A} + Q.
\end{equation}
The principle of the approach presented in this paper is  the MPC controller (\ref{MPC_1},\ref{MPC_2}) is implemented with $Q_N = P$. Proposition \ref{th:stability} summarises the relevant properties of the proposed MPC, building on classic MPC results from  \cite{Scokaert1998,Limon2003StableCM,Limon2009}.   
\begin{proposition} \label{th:stability}
Consider the MPC problem (\ref{MPC_2}) with $Q_N=P$, where $P$ and $K$ solve (\ref{LQR}-\ref{DARE}). Define $V^\star_{N}(x)$ as the optimal objective in (\ref{MPC_2}) with $x_t = x$. Denote the optimal stage cost with $x_t = x$ as $\ell(x, \hat{u}^\star_{0}(x))=x^\top Qx
 +\hat{u}^\star_{0}(x)^\top R \hat{u}^\star_{0}(x)$. Then, for the closed-loop system, it follows that:
    \begin{enumerate}
     \item For any $\bar{N}\geq1$, there exists a closed and bounded set, $\Omega_{\bar{N}}$, such that, if $x_0\in\Omega_{\bar{N}}$ and $\hat{w}_t=0,\ \forall t\geq0$, then the MPC solution is infinite-horizon optimal for any $N\geq\bar{N}$. 
        \item There exist positive scalars $d$, $\alpha$, such that, for any $N\geq1$, if $\hat{w}_t=0,\ \forall t\geq0$ then the MPC constraints are feasible, $\forall t\geq0$, and the origin is asymptotically stable $\forall x_0\in\Gamma_{N}$, with 
        \begin{equation}
            \Gamma_{N} = \left\{x\in\mathbb{R}^n: V^\star_{N}(x)\leq \ell(x, \hat{u}^\star_{0}(x))+(N-1)d + \alpha \right\}.
        \end{equation}
        \item There exist a scalar, $\mu\geq0$, such that, for any $N\geq1$ the MPC constraints are robustly feasible,  $\forall t\geq0$, and the system is Input-to-State Stable (ISS) $\forall x_0\in\Gamma_{N}$ given an additive model error, $\hat{w}$, such that: $\|\hat{w}_t\|\leq\mu,\ \forall t\geq0$. In other words: 
        $$V^\star_{N}(x_{t+d t})\leq V^\star_{N}(x_t)-\ell(x_t, \hat{u}^\star_{0}(x_t))+\sigma(\|\hat{w}_t\|),$$ 
        for some strictly increasing, bounded function, $\sigma(\cdot)$, with $\sigma(0)=0$. 
       
        \item The QP matrices, $H$, $M$ and the vector $q$, in (\ref{QP}), have finite norms for any $N\geq1$. 
    \end{enumerate}
\end{proposition}
\begin{proof}
Proof of Proposition \ref{th:stability} is given in Appendix  \ref{appendix_proofs}. \end{proof}

\paragraph{Implications.} Proposition \ref{th:stability} has some important implications. First, point 1 implies that there exists a state-dependant finite horizon length, $\bar{N}$, which is sufficient to make the MPC problem infinite-horizon optimal. This $\bar{N}$ can be upper bounded for a closed and bounded set of feasible states, $\Omega_{\bar{N}}$. \cite{Scokaert1998} proposed an iterative search that increases the horizon until optimality is verified; a similar algorithm is discussed in Appendix \ref{sec:horizon_reduction} where learning is completed with a large horizon and then iteratively reduced afterwards, although it is not implemented in this paper. Point 2,3 state that MPC that can provide stability and constraints satisfaction, hence \emph{safety}, if the model error is small. This also applies to small errors in the QP solution.  Finally, point 4 states that the QP matrices have finite norm when the system dynamics are pre-stabilised using the LQR gain\footnote{Note that any stabilising gain would be acceptable for the purpose of QP conditioning only.}, so the MPC problem is well conditioned and can be solved reliably to high accuracy, even over long horizons. If the open-loop system is unstable then the terms of the matrices in Appendix \ref{appendix_MPC_1} for the standard form are unbounded, so the QP solution may be poorly conditioned and the result inaccurate for long horizons. This can in turn invalidate the results of \cite{amos_optnet:_2017} which assumes that the KKT conditions are exactly satisfied in order to compute its gradients.

\paragraph{DARE Derivative.}

In order to implement $Q_N = P$ in a differentiable imitation learning framework such as that presented in Section \ref{section_diff_mpc}, the solution of the DARE is differentiated as follows.

\begin{proposition}\label{proposition_DARE_diff}
Let $P$ be the stabilizing solution of (\ref{DARE}), and assume that $Z_1^{-1}$ and $(R + B^\top P B)^{-1}$ exist, then the Jacobians of the implicit function defined by (\ref{DARE}) are given by
$$
\frac{\partial \vec P}{\partial \vec A} = Z_1^{-1} Z_2, \quad \frac{\partial \vec P}{\partial \vec B} = Z_1^{-1} Z_3, \quad \frac{\partial \vec P}{\partial \vec Q} = Z_1^{-1} Z_4, \quad \frac{\partial \vec P}{\partial \vec R} = Z_1^{-1} Z_5,
$$
where $Z_1,\dots,Z_5$ are defined by
\begin{equation*}
\begin{aligned}
Z_1 &:= \mathbf{I}_{n^2} - (A^\top \otimes A^\top ) \big[ \mathbf{I}_{n^2} - (PBM_2 B^\top \otimes \mathbf{I}_n ) - (\mathbf{I}_n \otimes PBM_2 B^\top ) \\
& \hspace{150pt} + (PB \otimes PB) (M_2 \otimes M_2) (B^\top \otimes B^\top ) \big] \\
Z_2 &:= ( \mathbf{V}_{n,n} + \mathbf{I}_{n^2}) (\mathbf{I}_n \otimes A^\top M_1) \\
Z_3 &:= (A^\top \otimes A^\top ) \big[ (PB \otimes PB) (M_2 \otimes M_2 )(\mathbf{I}_{m^2} + \mathbf{V}_{m,m})(\mathbf{I}_{m} \otimes B^\top P) \\
& \hspace{205pt} - (\mathbf{I}_{n^2} + \mathbf{V}_{n,n})(PBM_2 \otimes P) \big] \\
Z_4 &:= \mathbf{I}_{n^2} \\
Z_5 &:=  (A^\top \otimes A^\top ) (PB \otimes PB) (M_2 \otimes M_2 ),
\end{aligned}
\end{equation*}
and $M_1,M_2,M_3$ are defined by
$$
M_1 := P - P B M_2 B^\top P, \quad M_2 := M_3^{-1}, \quad M_3 := R + B^\top P B.
$$
\end{proposition}
\begin{proof} The proof of Proposition \ref{proposition_DARE_diff} is given in Appendix \ref{appendix_DARE_derivative}.
\end{proof}

The sensitivity of the DARE solution has been investigated in the context of robustness to perturbations in the input matrices, e.g. \cite{riccati_sensitivity_Sun, konstantinov1993perturbation}, and the analytical derivative of the continuous-time algebraic Riccati equation was derived in \cite{riccati_derivative_brewer} by differentiating the exponential of the Hamiltonian matrix, but to the best of the authors' knowledge this is the first presentation of an analytic derivative of the DARE using the differential calculus approach of \cite{magnus99}.

\begin{wrapfigure}{R}{7cm}
    \noindent\begin{minipage}{0.5\columnwidth}
      \vspace{-1.1cm}
        \begin{algorithm}[H]
          \DontPrintSemicolon
            \KwInput{$\mathcal{M} \setminus \mathcal{S}$, $N>0$, $\beta\geq0$, $N_{\text{e}}>0$. \textbf{Out:} $\mathcal{S}$ }
            %\KwOutput{$\mathcal{S}$}
            \small
            \caption{Infinite-horizon MPC Learning}
            \label{alg:alternateDescent}
            \For{$i=0...N_{\text{e}}$}{ 
                \textbf{Forward Pass} \;
                
                $(K,\ P)\leftarrow$ \text{DARE (\ref{LQR}-\ref{DARE}) solution} \;
                
                $Q_T \gets P$ \;
                
                $\hat{u}_{0:N}^\star \gets$ \text{MPC QP  (\ref{stable_system_dynamics}-\ref{QP}) solution} \;
                
                $L\gets$  \text{Imitation loss  (\ref{eq:imitation_loss_mixed})}\; 
                
                \textbf{Backward Pass} \;
                
                \text{Differentiate loss (\ref{eq:imitation_loss_mixed})} \;
                
                \text{Differentiate MPC QP solution, $\hat{u}_{0:N}^\star$,}\; \text{using Appendix \ref{appendix_OSQP} }\;
                
                \text{Differentiate DARE, $(P, K)$,}\; 
                \text{using Proposition \ref{proposition_DARE_diff}} \;
                
                \textbf{Update step} \;
                
                $\mathcal{S}\gets$ \text{Gradient-based step} \;
            }
        \end{algorithm}
    \end{minipage}
    \vspace{-1.0cm}
 \end{wrapfigure}

\paragraph{Algorithm overview}
Algorithm \ref{alg:alternateDescent} presents the overall procedure for learning a subset, $\mathcal{S}$, of the MPC controller parameters, $\mathcal{M} = \{A, B, Q, R, \underline{x}, \overline{x}, \underline{u}, \overline{u}, k_u, k_x \}$, with the key steps of the forwards and backwards pass of a gradient-based optimization method. In each forward pass the MPC terminal cost matrix, $Q_N$, and the pre-stabilizing controller, $K$, are set from the solution of the DARE, then the DARE and MPC QP solutions are differentiated in the backward pass to obtain the gradients. Note that the horizon, $N$, is not differentiable, and that learning the entire set $\mathcal{M}$ simultaneously is challenging in general.

%\newpage

\section{Numerical Experiments}\label{sec:experiments}
In this section the performance of the algorithm was demonstrated through numerical experiments in two test cases: firstly on a set of second order mass-spring-damper models to provide a performance baseline in an easily interpretable setting, and then on vehicle platooning problem to investigate a higher-dimensional real-world application.

\subsection{Mass-Spring-Damper}\label{sec_MSD}

\paragraph{Model \& Expert}
Expert data was generated using a mass-spring-damper model parameterized by a mass, $m \in \mathbb{R} > 0$, damping coefficient, $c \in \mathbb{R}$, stiffness, $k \in \mathbb{R}$, and timestep $dt\in \mathbb{R} > 0$, where
$$
A = \text{exp}(A_c dt), \quad A_c = \begin{bmatrix} 0 & 1 \\ -\frac{k}{m} & -\frac{c}{m} \end{bmatrix}, \quad B = (A - I_n)A_c^{-1}B_c, \quad B_c = \begin{bmatrix} 0 \\ \frac{1}{m} \end{bmatrix},
$$
so that $x_t \in \mathbb{R}^2$ is the position and velocity of the mass, and the $u_t \in \mathbb{R}$ is a force applied to the mass. 

\begin{wraptable}{R}{0.55\textwidth}
  \vspace{-0.4cm}
  \caption{Damping coefficient $c$ used to generate the seven imitation systems.}
  \label{coefficient-table}
  \centering
  \begin{tabular}{lccccccc}
    \toprule
    System     & $1$ & $2$ & $3$ & $4$ & $5$ & $6$ & $7$ \\
    \midrule
    $c$ & 1 & 0.5 & 0.1 & -0.1 & -0.3 & -0.5 & -0.6 \\
    \bottomrule
  \end{tabular}
%   \vspace{0}
\end{wraptable}

Seven models were created with $m=1$, $k=1$, and $dt = 0.2$, and $c$ was varied as shown in Table \ref{coefficient-table} to affect the open-loop stability of the models ($c>0 \Rightarrow$ stable, $c<0 \Rightarrow$ unstable). The expert data was then generated by simulating each of the systems the initial condition $x_0 = (0,3)$ in closed-loop with an infinite-horizon MPC controller (i.e. the horizon was increased until the open-loop state predictions matched the closed-loop response), using $Q = \text{diag}([1,1])$, $R = 2$, $(\underline{u},\overline{u}) = (-\infty,0.5)$, $\underline{x} = (-1, -\infty)$, and $\overline{x}) = (1, \infty)$. The constraint set was chosen so that the constraints on both state and control input were strongly active at the solution whilst ensuring that the expert MPC optimization was feasbile. The values $k_u = k_x = 100$ were found to be sufficient to enforce the hard constraints and were used for all experiments. It is important to note that the approach of \citep{amos_differentiable_2018} cannot be used reliably for even this simple example as it does not consider state constraints, and when hard constraints are added to the method it fails in general because the optimization problem has become infeasible in the forwards pass at some time $t$. 
% For a large enough $N$, our method provides constraint satisfaction $\forall t$.  

\begin{wrapfigure}{R}{0.65\textwidth}
\vspace{-0.65cm}
    \centering
    \includegraphics[scale=0.83]{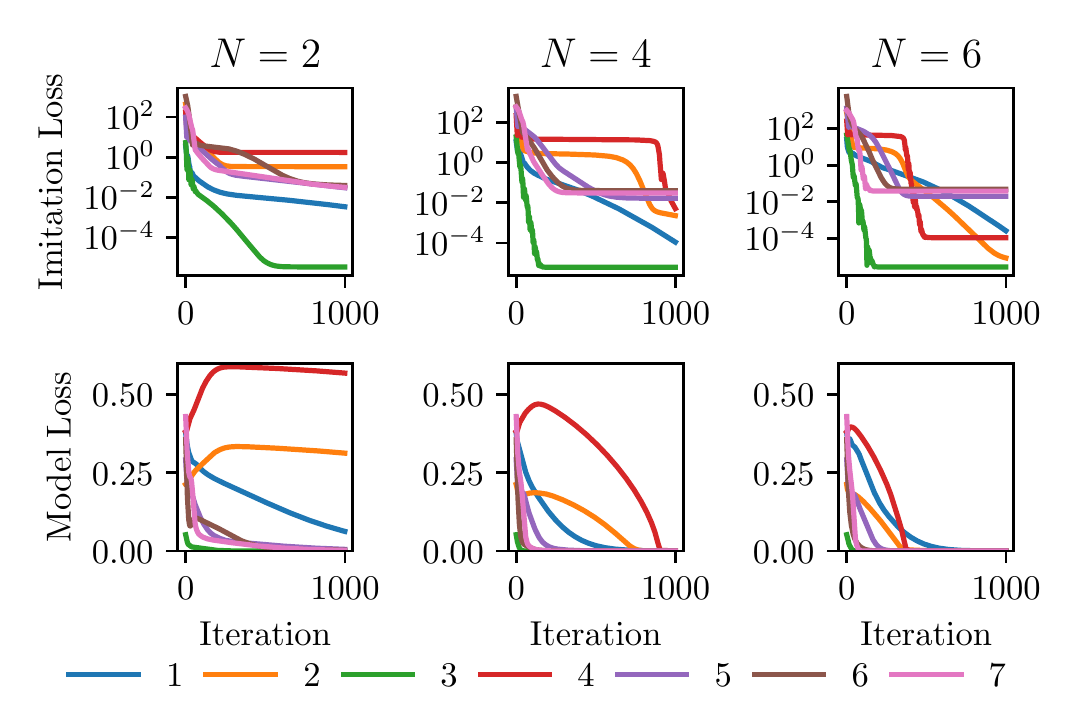}
    \caption{{\bf Mass-spring-damper. Imitation loss and model loss at each iteration of the training process.} Top row: imitation loss. Bottom row: model loss given by $\|\v A - \v A_j \|^2_2$, where $A_j$ is the learned model at iteration $j$, and $A$ is the correct model. The model loss was \textit{not} used as part of the training process, and shown only to indicate whether the model is converging correctly.}
    \label{fig:loss}
    \vspace{-0.35cm}
\end{wrapfigure}

\paragraph{Learning}
The learner and expert shared all system and controller information apart from the state transition matrix $A$, which was learned, and the MPC horizon length, which was implemented as each of $N \in \{ 2,3,6 \}$ in three separate experiments. $A$ was initialized with the correct state transition matrix plus a uniformly distributed pseudo-random perturbation in the interval $[-0.5,0.5]$ added to each element. The learner was supplied with the first 50 elements of the closed loop state trajectory and corresponding controls as a batch of inputs, and was trained to minimize the imitation loss (\ref{eq:imitation_loss_mixed}) with $\beta = 0$, i.e. the state dynamics were learned using predicted control trajectories \textit{only}, and the state transitions are not made available to the learner (this is the same approach used in \citealp{amos_differentiable_2018}). The experiments were implemented in Pytorch 1.2.0 using the built-in Adam optimizer \citep{kingma2015a} for 1000 steps using default parameters. The MPC optimization problems were solved for the `expert' and `learner' using OSQP \citep{osqp} with settings (eps\_abs=1E-10, eps\_rel=1E-10, eps\_rim\_inf=1E-10, eps\_dual\_inf=1E-10).

% \vspace{-0.2cm}
\paragraph{Results} 
\begin{comment}
\begin{wrapfigure}{R}{0.65\textwidth}
\vspace{-0.65cm}
    \centering
    \includegraphics[scale=0.83]{fig_model_loss.pdf}
    \caption{{\bf Mass-spring-damper. Imitation loss and model loss at each iteration of the training process.} Top row: imitation loss. Bottom row: model loss given by $\|\v A - \v A_j \|^2_2$, where $A_j$ is the learned model at iteration $j$, and $A$ is the correct model. The model loss was \textit{not} used as part of the training process, and shown only to indicate whether the model is converging correctly.}
    \label{fig:loss}
    \vspace{-0.35cm}
\end{wrapfigure}
\end{comment}

Figure \ref{fig:loss} shows the imitation and model loss at each of the 1000 optimization iterations for each of the tested horizon lengths. It can be seen that for all of the generated systems the imitation loss converges to a low value, although this is a local minimum in general. In most cases, the learned model converges to a close approximation of the real model, although as the problem is non-convex this cannot be guaranteed, and it is also shown that there are some cases in which the model does not converge correctly. This occurred exclusively for $N=2$, where neither system $4$ nor system $2$ converge to the correct dynamics. Additionally, it can be seen that both the imitation loss and model loss converge faster as the prediction horizon is increased. This suggests that a longer learning horizon improves the learning capabilities of the methods, but there is not sufficient data to demonstrate this  conclusively. 
%\newpage

\begin{wrapfigure}{r}{0.65\textwidth}
\vspace{-0.7cm}
    \centering
    \includegraphics[scale=0.83]{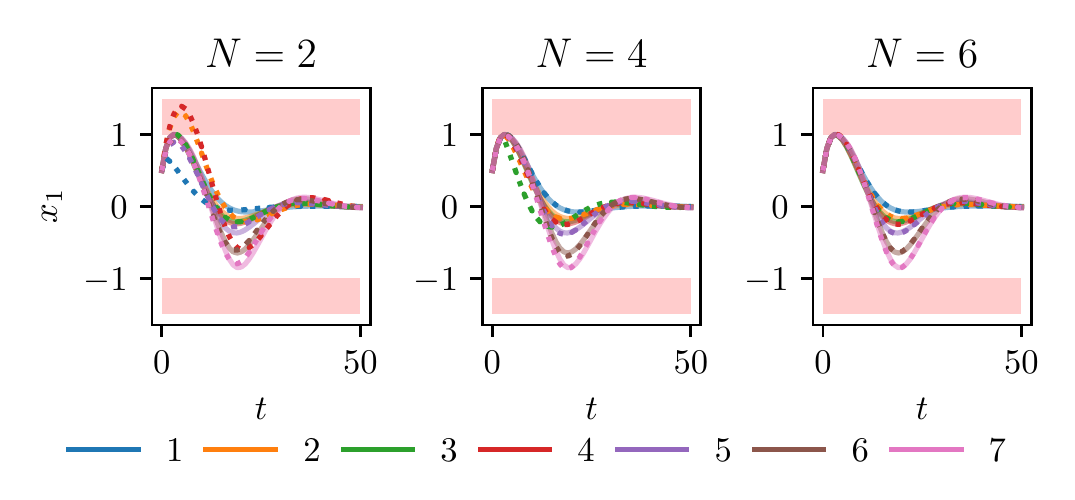}
    \caption{\textbf{Mass-spring-damper. Closed-loop trajectories using the expert and learned controllers.} Trajectories only shown for $x_1$ (position), but $x_2$ (velocity) can be inferred. Expert controllers shown with solid lines, and learned controller shown with dotted lines. The hard state constraints are shown in the red regions.\label{fig:results}}
\vspace{-0.25cm}
\end{wrapfigure}

To test generalization performance, each of the systems was re-initialized with initial condition $x_0 = (0.5,2)$ and simulated in closed loop using the learned controller for each horizon length. The results are compared in Figure \ref{fig:results} against the same systems controlled with an infinite horizon MPC controller. The primary observation is that as the learned MPC horizon is increased to $N = 6$, the closed loop trajectories converge to expert trajectories, indicating that the infinite horizon cost has been learned (when using the infinite horizon cost with no model mismatch or disturbance, the predicted MPC trajectory is exactly the same as the closed loop trajectory), and that the state constraints are guaranteed for $N \geq 4$. Furthermore, it can be seen that the learned controllers are stabilizing, even for the shortest horizon and the most unstable open-loop systems. This is also the case for systems $2$ and $4$ where the incorrect dynamics were learned, although in this case the state constraints are not guaranteed for $N=2$.

\subsection{Vehicle Platooning}\label{sec_main_vehicle_platoon}

\begin{wrapfigure}{r}{0.50\textwidth}
\vspace{-0.6cm}
    \centering
%% Accompanies image file 'car_diagram_scaled.pdf' (pdf, eps, ps)
\begingroup%
  \makeatletter%
  \providecommand\color[2][]{%
    \errmessage{(Inkscape) Color is used for the text in Inkscape, but the package 'color.sty' is not loaded}%
    \renewcommand\color[2][]{}%
  }%
  \providecommand\transparent[1]{%
    \errmessage{(Inkscape) Transparency is used (non-zero) for the text in Inkscape, but the package 'transparent.sty' is not loaded}%
    \renewcommand\transparent[1]{}%
  }%
  \providecommand\rotatebox[2]{#2}%
  \newcommand*\fsize{\dimexpr\f@size pt\relax}%
  \newcommand*\lineheight[1]{\fontsize{\fsize}{#1\fsize}\selectfont}%
  \ifx\svgwidth\undefined%
    \setlength{\unitlength}{218.91610453bp}%
    \ifx\svgscale\undefined%
      \relax%
    \else%
      \setlength{\unitlength}{\unitlength * \real{\svgscale}}%
    \fi%
  \else%
    \setlength{\unitlength}{\svgwidth}%
  \fi%
  \global\let\svgwidth\undefined%
  \global\let\svgscale\undefined%
  \makeatother%
  \begin{picture}(1,0.16342989)%
    \lineheight{1}%
    \setlength\tabcolsep{0pt}%
    \put(0,0){\includegraphics[width=\unitlength,page=1]{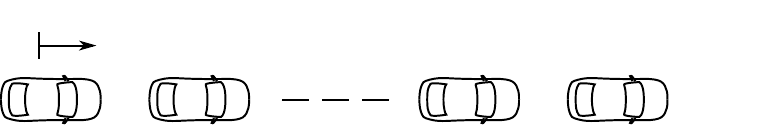}}%
    \put(0.08624992,0.12124763){\color[rgb]{0,0,0}\makebox(0,0)[lt]{\lineheight{1.25}\smash{\begin{tabular}[t]{l}$y_1$\end{tabular}}}}%
    \put(0,0){\includegraphics[width=\unitlength,page=2]{car_diagram_scaled.pdf}}%
    \put(0.28157078,0.12124763){\color[rgb]{0,0,0}\makebox(0,0)[lt]{\lineheight{1.25}\smash{\begin{tabular}[t]{l}$y_2$\end{tabular}}}}%
    \put(0,0){\includegraphics[width=\unitlength,page=3]{car_diagram_scaled.pdf}}%
    \put(0.63669956,0.12124763){\color[rgb]{0,0,0}\makebox(0,0)[lt]{\lineheight{1.25}\smash{\begin{tabular}[t]{l}$y_{n_v-1}$\end{tabular}}}}%
    \put(0,0){\includegraphics[width=\unitlength,page=4]{car_diagram_scaled.pdf}}%
    \put(0.83202039,0.12124763){\color[rgb]{0,0,0}\makebox(0,0)[lt]{\lineheight{1.25}\smash{\begin{tabular}[t]{l}$y_{n_v}$\end{tabular}}}}%
    \put(-0.21164732,0.50681195){\color[rgb]{0,0,0}\makebox(0,0)[lt]{\begin{minipage}{1.40464764\unitlength}\raggedright \end{minipage}}}%
  \end{picture}%
\endgroup%

    \caption{\textbf{Platoon Model.} $n_v$ vehicles in 1 degree of freedom where $y$ is longitudinal displacement.\label{car_diagram}}
\vspace{-0.3cm}
\end{wrapfigure}

\paragraph{Model \& Expert} Vehicle platoon control is a problem that has been studied using control theory (e.g. \cite{platoon}), but here it is demonstrated that a safe, stabilizing controller can be learned from examples of vehicles driving in formation. Figure \ref{car_diagram} shows an illustration of a platoon of $n_v$ vehicles for which the objective is to stabilize the relative longitudinal positions of each vehicle to the steady-state conditions $y_i - y_{i-1} \to y_{ss}$ and $\dot{y}_i - \dot{y}_{i-1} \to 0$ $\forall i$, subject to the hard constraint that relative position of the vehicles is never lower than a safe threshold $y_i - y_{i-1} \geq \underline{y}$ $\forall{i}$, and that the vehicles' ability to brake and accelerate is constrained by $b \leq \ddot{y}_i \leq a$ $\forall i$ where $b < 0 < a$ (note that only the relative positions and velocities of the vehicles is considered, as the global position and velocity of the platoon can be controlled separately by adding an equal perturbation to each element of $\ddot{y}$). In appendix \ref{sec_plat_model} it is shown that this can be modelled as a discrete time LTI system.  
% by
% \begin{figure}
% \begin{center}
% \input{car_diagram.pdf_tex}
% \caption{Diagram of platoon of $n_v$ vehicles in 1 degree of freedom.}
% \label{car_diagram}
% \end{center}
% \vspace{-0.35cm}
% \end{figure}
% 
% \begin{equation}\label{discrete_time_platoon}
% x_{t + d_t} = \begin{bmatrix}
% \mathbf{I} & dt \mathbf{I}  \\
% \mathbf{O} & \mathbf{I}
% \end{bmatrix} x_t + \begin{bmatrix}
% \frac{1}{2}  \hat{B}_c (dt)^2 \\ \hat{B}_c dt
% \end{bmatrix} u_t, \quad \hat{B}_c = \begin{bmatrix}
% -1 & 1 \\
% & \ddots & \ddots \\
% & & -1 & 1
% \end{bmatrix}
% \end{equation}
%
% where $\hat{B}_c$, $x_t \in \mathbb{R}^{2(n_v -1)}$, and $u \in \mathbb{R}^{n_v}$ are defined in appendix \ref{sec_plat_model}, and are subject to the constraints 
% $$
% x_t \geq \begin{bmatrix}
% (\underline{y} - y_{ss}) \mathbf{1}_{n_v-1} \\ - \infty 
% \end{bmatrix}, \quad \text{and} \quad b \textbf{1} \leq u_t \leq a \textbf{1} \quad \forall t.
% $$
Expert data was generated from the model with  $n_v = 10$ vehicles so that $x_t \in \mathbb{R}^{18}$ and $u_t \in \mathbb{R}^{10}$. 20 instances were generated using random feasible initial conditions with $y_{ss} = 30\ $m and $\underline{y} = 10\ $m, and then simulated for $20 \ $s in time intervals of $dt = 0.7 \ $s with an infinite-horizon MPC controller, using $Q = \mathbf{I}_n$ and $R = 2\mathbf{I}_m$.

\vspace{-0.2cm}
\paragraph{Learning}
The learner and expert shared all system and controller information apart from the cost matrices $Q$ and $R$, which were learned, and the MPC horizon length, which was implemented as each of $N \in \{ 5,10,15,20\}$ in four separate experiments. The matrices $Q$ and $R$ were initialized as completely random diagonal matrices with each element uniformly distributed in the interval $[0,3]$, and the diagonal structure was maintained through training. 500 training iterations were used; otherwise the learning process (loss function, learning rate, etc...) was the same as in Section \ref{sec_MSD}. 

\begin{wrapfigure}{R}{0.57\textwidth}
\vspace{-0.2cm}
\centering
    \includegraphics[scale=0.83]{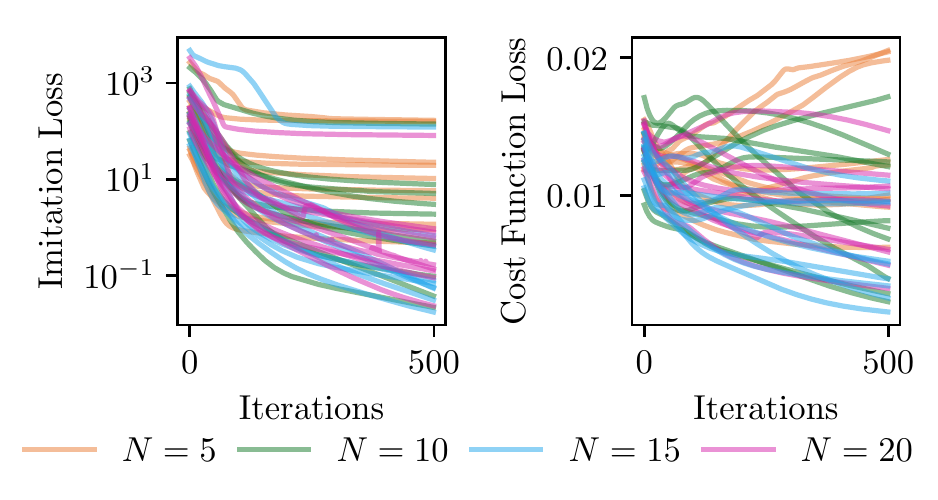}
    \vspace{-0.5cm}
    \caption{{\bf Vehicle platooning. Imitation loss and cost function loss at each iteration of the training process.} Left: imitation loss. Right: model loss given by $\| \vec Q - \vec Q_j\|^2_2 + \| \vec R - \vec R_j\|^2_2$, where $Q$ and $R$ are the correct cost matrices and $Q_j$ and $R_j$ are the cost matrices at iteration $j$.}
    \label{fig:loss2}
    \vspace{-0.2cm}
\end{wrapfigure}
\vspace{-0.2cm}
\paragraph{Results}
Figure \ref{fig:loss2} shows the imitation and cost function losses at each of the 500 optimization iterations for each of the tested horizon lengths and initial conditions. As with the mass-spring-damper experiments, it is suggested that a longer prediction horizon improves training as the imitation loss generally converges to a lower value for the examples with $N \in \{15,20\}$, but only convergence to a local minimum is achieved in general. The cost error also does not converge in general (although better convergence is observed again for the longer horizon lengths), however for this learning problem there is a manifold of matrices $Q$ and $R$ with the same minimizing argument, so divergence of the cost error does not necessarily indicate that the learned cost function is `incorrect'. Furthermore, in this case the model is known exactly, so the closed-loop infinite-horizon properties can be obtained even without the correct cost function. 

\begin{figure}[b!]
\centering
    \vspace{-0.5cm}
    \includegraphics[scale=0.83, clip]{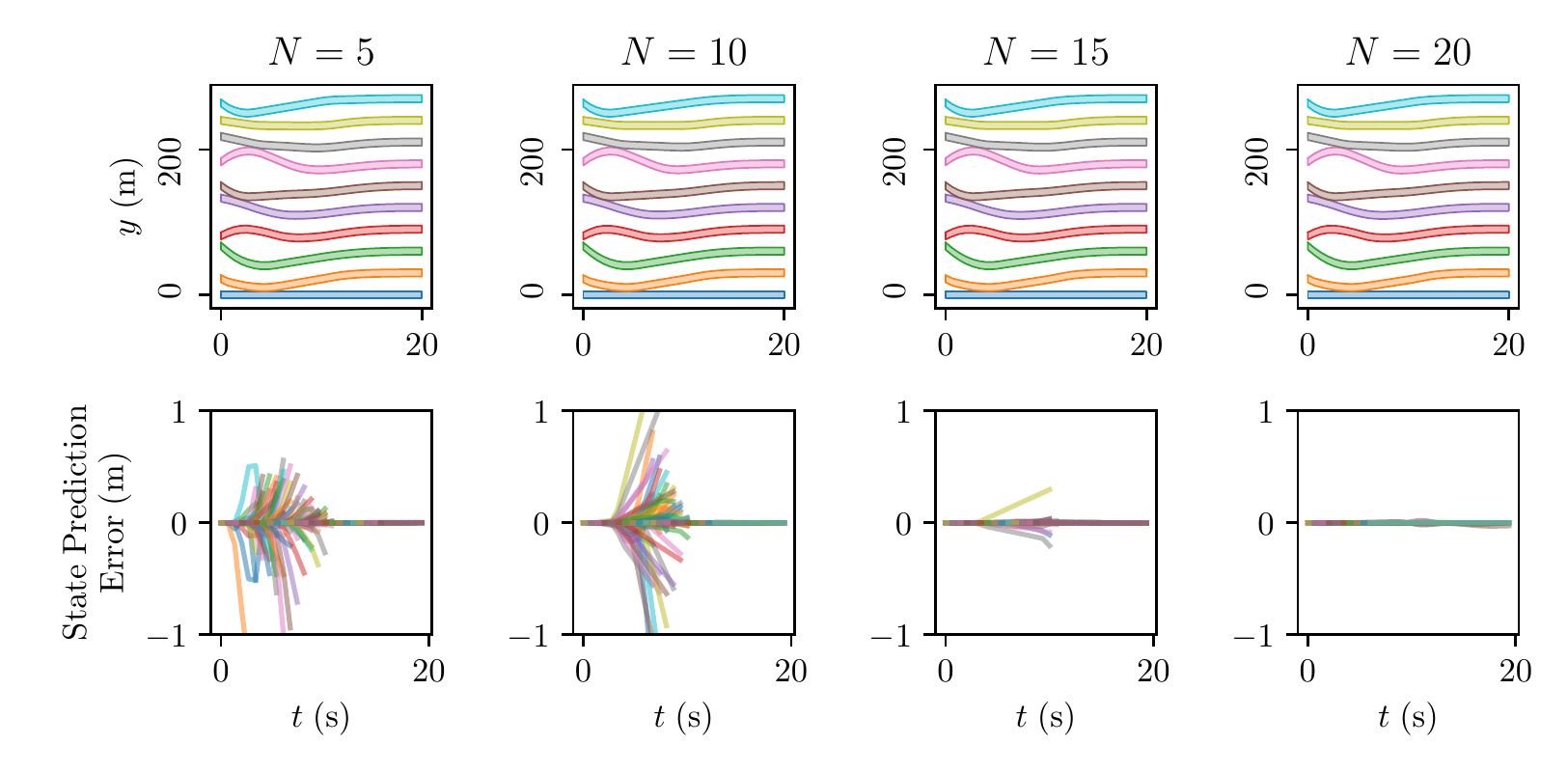}
    \vspace{-0.3cm}
    \caption{{\bf Vehicle platooning. Closed loop simulation and prediction error for all horizon lengths.} Top row: closed loop simulation where each shaded region is the safe separation distance for each vehicle. Bottom row: prediction error given by $x_{[t:t+N]} - \hat{x}_t $, where $\hat{x}$ is the state trajectory predicted by the MPC at time $t$.}
    \label{fig:platoon_training}
    \vspace{-0.cm}
\end{figure}

Figure \ref{fig:platoon_training} shows the model simulated from the same initial condition in closed loop using a learned controller for each of the horizon lengths, together with the error between the MPC state predictions and ensuing closed-loop behaviour. All of the controllers are observed to successfully satisfy the hard constraints on vehicle separation, and all converge to the correct steady-state vehicle separation. The differences between the prediction capabilities of the controllers is highlighted by the state prediction errors, and it can be seen that for $N=20$ the state predictions match the ensuing behaviour, indicating that the infinite horizon cost is being used and that closed-loop stability is guaranteed, even without the use of a terminal constraint set. It is also demonstrated for $N < 20$  that the largest errors occur from predictions made at times when the state constraints are active, suggesting that these controllers deviate from their predictions to satisfy the constraints at later intervals.

\vspace{-0.2cm}
\subsection{Limitations}\label{sec:limitations}

The above approach is limited in scope to LTI systems, and a more comprehensive solution would cover linear time varying systems (for which the MPC is still obtained from the solution of a QP). In this case the infinite horizon cost cannot be obtained from the solution of the DARE, and the extension of the methods presented in this paper to time varying or non-linear models is non-trivial (see Appendix \ref{app:nonlinear} for further discussion). Additionally, the derivative of the DARE in Proposition \ref{proposition_DARE_diff} involves multiple Kronecker products and matrix inversions (including an $n^2 \times n^2$ matrix) that do not scale well to large state and control dimensions, although the dynamics of physical systems can usually be reasonably approximated with only a few tens of variables, so this may not become an issue in practice. The algorithm also requires a stabilizing solution of the DARE to exist; theories for the existence of stabilizing solutions are non-trivial (e.g. \citealp{RAN198863}), and it is not immediately obvious how to enforce their existence throughout the training process (stabilizibility can be encouraged using the one-step ahead term in (\ref{eq:imitation_loss_mixed})).

\begin{comment}
\section{Conclusion}
This work presents a method to differentiate through an infinite-horizon linear quadratic MPC, where the solution of the DARE was used to compute a terminal cost from the MPC optimization problem. The final control sequence is obtained from the solution of a QP that is structured so that its always both well-conditioned and feasible, and the whole forward pass is end-to-end differentiable,  so can be included as a layer in a neural network architecture. The approach is demonstrated on two sets of imitation learning experiments, it is shown that a short prediction horizon can be found such that the resulting MPC is stabilizing and infinite-horizon optimal.      
\end{comment}

\section*{Acknowledgments}
The authors are grateful to Brandon Amos for providing support using his differentiable QP tool (\texttt{https://github.com/locuslab/optnet}) in the preliminary work for this project (all of the methods presented in this paper were developed independently).   

\bibliographystyle{iclr2020_conference}
\bibliography{bibli.bib}

\newpage

\section*{Appendices}
\appendix
\section{MPC quadratic program}\label{appendix_MPC_1}
Problem (\ref{MPC_1}) is equivalent to 
\begin{equation*}
\begin{aligned}
z^{\star} = \underset{z}{\text{argmin}} \ & \frac{1}{2} z^\top \begin{bmatrix} \mathbf{R} + \Psi^\top \mathbf{Q} \Psi \\ & \mathbf{O}_{Nm \times Nm} \\ && \mathbf{O}_{Nn \times Nn} \end{bmatrix}  z +  \begin{bmatrix} \Psi \mathbf{Q} \Phi x_t \\ k_u \mathbf{1}_{Nm} \\ k_x \mathbf{1}_{Nn} \end{bmatrix}^\top z  \\
\text{s.t.} \ & \begin{bmatrix} \underline{\mathbf{u}} \\ -\infty \\ 0_{Nm} \\ \underline{\mathbf{x}} - \Phi x_t \\ -\infty \\ 0_{Nn}  \end{bmatrix} \leq \begin{bmatrix} \mathbf{I}_{Nm} & \mathbf{I}_{Nm} & \\ \mathbf{I}_{Nm} & -\mathbf{I}_{Nm} & \\ & \mathbf{I}_{Nm} & \\ \Psi & & \mathbf{I}_{Nn} \\ \Psi & & -\mathbf{I}_{Nn} \\ & & \mathbf{I}_{Nn} \end{bmatrix} z \leq  \begin{bmatrix} \infty \\ \overline{\mathbf{u}} \\ \infty \\ \infty  \\ \overline{\mathbf{x}} - \Phi x_t \\ \infty \end{bmatrix},
\end{aligned}
\end{equation*}
where
\begin{gather*}
z = \begin{bmatrix} \hat{u} \\ r \\ s \end{bmatrix}, \quad \mathbf{R} = \begin{bmatrix} R \\ & \ddots \\ & & R \end{bmatrix}, \quad \mathbf{Q} = \begin{bmatrix} Q \\ & \ddots \\ & & Q \\ & & & Q_N \end{bmatrix}, \quad \Phi = \begin{bmatrix} A \\ \vdots \\ A^{N} \end{bmatrix}, \quad \\ \\
\Psi = \begin{bmatrix} B \\ \vdots & \ddots \\ A^{N-1} B & \cdots & B \end{bmatrix}, \quad 
\underline{\mathbf{x}} = \begin{bmatrix} \underline{x} \\ \vdots \\ \underline{x} \end{bmatrix}, \quad \overline{\mathbf{x}} = \begin{bmatrix} \overline{x} \\ \vdots \\ \overline{x} \end{bmatrix}, \quad \underline{\mathbf{u}} = \begin{bmatrix} \underline{u} \\ \vdots \\ \underline{u} \end{bmatrix}, \quad \overline{\mathbf{u}} = \begin{bmatrix} \overline{u} \\ \vdots \\ \overline{u} \end{bmatrix},
\end{gather*}
are of conformal dimensions. Using the above, problem (\ref{MPC_2}) is then equivalent to
\begin{equation*}
\begin{aligned}
z^{\star} = \underset{z}{\text{argmin}} \ & \frac{1}{2} z^\top \begin{bmatrix} (\mathbf{K} \hat{\Psi} + \mathbf{I}_{Nm})^\top \mathbf{R} (\mathbf{K} \hat{\Psi} + \mathbf{I}_{Nm}) + \hat{\Psi}^\top \hat{\mathbf{Q}} \hat{\Psi} \\ & \mathbf{O}_{Nm \times Nm} \\ && \mathbf{O}_{Nn \times Nn} \end{bmatrix}  z \\
& +  \begin{bmatrix} (\mathbf{K}^\top \mathbf{R} (\mathbf{K} \hat{\Psi} + \mathbf{I}_{Nm} ) + \hat{\mathbf{Q}} \hat{\Psi} )^\top \hat{\Phi} x_t \\ k_u \mathbf{1}_{Nm} \\ k_u \mathbf{1}_{Nn} \end{bmatrix}^\top z  \\
\text{s.t.} \ & \begin{bmatrix}  \underline{\mathbf{u}} - \mathbf{K} \hat{\Phi} x_t \\ -\infty \\ \mathbf{0}_{Nm} \\ \underline{\mathbf{x}} - \Phi x_t \\ -\infty \\ \mathbf{0}_{Nn}  \end{bmatrix} \leq \begin{bmatrix} (\mathbf{K} \hat{\Psi} + \mathbf{I}_{Nm}) & \mathbf{I}_{Nm} & \\ (\mathbf{K} \hat{\Psi} + \mathbf{I}_{Nm}) & -\mathbf{I}_{Nm} & \\ & \mathbf{I}_{Nm} & \\ \Psi & & \mathbf{I}_{Nn} \\ \Psi & & -\mathbf{I}_{Nn} \\ & & \mathbf{I}_{Nn} \end{bmatrix} z \leq  \begin{bmatrix} \infty \\ \overline{\textbf{u}} - \mathbf{K} \hat{\Phi} x_t \\ \infty \\ \infty  \\ \overline{\mathbf{x}} - \Phi x_t \\ \infty \end{bmatrix},
\end{aligned}
\end{equation*}
where now
$$
z = \begin{bmatrix} \delta \hat{u} \\ r \\ s \end{bmatrix}, \quad \Phi = \begin{bmatrix} (A + BK) \\ \vdots \\ (A + BK)^{N} \end{bmatrix} \quad \text{and} \quad \Psi = \begin{bmatrix} B \\ \vdots & \ddots \\ (A+ BK)^{N-1} B & \cdots & B \end{bmatrix},
$$
and
\begin{align*}
\hat{\mathbf{Q}} = \begin{bmatrix} \mathbf{O}_{n \times n} \\ & \mathbf{Q} \end{bmatrix}, \quad \hat{\Phi} = \begin{bmatrix} \mathbf{I}_n \\ \Phi \end{bmatrix}, \quad \hat{\Psi} = \begin{bmatrix} \mathbf{O}_{n \times Nn} \\ \Psi \end{bmatrix}, \quad  \mathbf{K} = \begin{bmatrix} K \\ & \ddots & & \mathbf{O}_{Nm \times n} \\ & & K & \end{bmatrix},
\end{align*}
are of conformal dimensions.

\section{OSQP derivatives}\label{appendix_OSQP}
OSQP solves quadratic programs of the form (\ref{QP}), and returns values for $z$, $y$, and $s$ that satisfy
\begin{align*}
& Mz = s, \\
& Hz + q - M^\top y = 0, \\
& s \in \mathcal{C}, \quad y \in \mathcal{N}_\mathcal{C}(s),
\end{align*}
\cite[\S 2]{osqp}, where $\mathcal{C}$ is the set $\{ s: l \leq s \leq u \}$, and $\mathcal{N}_\mathcal{C}$ is the normal cone of $\mathcal{C}$. The values of $y$ that are returned by the solver can be used to determine whether the constraints are strongly active at the solution, where $y_i = 0$ indicates that the constraints $l_i \leq M_i z$ and $M_i z \leq u_i$ are inactive, $y_i > 0$ indicates that $M_i z \leq u_i$ is strongly active, and $y_i < 0$ indicates that $l_i \leq M_i z$ is strongly active. The solution can therefore be completely characterised by the KKT system
\begin{equation}\label{KKT_LSE}
\begin{bmatrix}
H & M_\mathcal{U}^\top & M_\mathcal{L}^\top \\
M_\mathcal{U} \\
M_\mathcal{L}
\end{bmatrix}
\begin{bmatrix}
z \\ y_\mathcal{U} \\ y_\mathcal{L}
\end{bmatrix} = \begin{bmatrix}
q \\ u_\mathcal{U} \\ l_\mathcal{L}
\end{bmatrix}
\end{equation}
where $\mathcal{U} = \{ i:y_i>0 \}$ and $\mathcal{L} = \{ i:y_i < 0 \}$, and the notation $M_\mathcal{S}$ indicates a matrix consisting of the $i \in \mathcal{S}$ columns of given matrix $M$, and $v_\mathcal{S}$ indicates a vector consisting of the $i \in \mathcal{S}$ elements of given vector $v$. Equation (\ref{KKT_LSE}) can then be differentiated using the techniques detailed in \cite[\S 3]{amos_optnet:_2017}.

\section{Proof of Proposition \ref{th:stability}}\label{appendix_proofs}
\begin{proof}({\bf Proposition \ref{th:stability}})  
The first point follows from \citep{Scokaert1998}. The next two points of Proposition \ref{th:stability} stem from the results in \citep{Limon2003StableCM,Limon2009}. In particular, the closed-loop is Lipschitz since the model is linear and the controller is the solution of a strictly convex QP. Moreover, the LQR provides a contractive terminal set.   The final point follows from the fact that $(A+BK)^N$ has eigenvalues in the unit circle, $\forall N\geq 1$. Proof of point 4 is concluded by inspection of the QP matrices (Appendix \ref{appendix_MPC_1}) and by application of Theorem 5.6.12, page 298 of \cite{Horn:2012:MA:2422911} which states that, given a bound, $\overline{\rho}$, on the spectral radius, then there exists a matrix norm which is also less than $\overline{\rho}$.     
\end{proof}

\section{Proof of Proposition \ref{proposition_DARE_diff}} \label{appendix_DARE_derivative}
\begin{proof} ({\bf Proposition \ref{proposition_DARE_diff}})  
If a stabilizing solution ($\rho(A+BK)\leq 1$) to (\ref{DARE}) exists, it is unique \cite[Proposition 1]{IONESCU1992229}, and the DARE can therefore be considered an implicit function of $A$, $B$, $Q$, and $R$. Using the assumption that $(R + B^\top P B)^{-1}$ exists, it can be concluded that $Z_1, \dots, Z_5$ and $M_1,M_2,M_3$ exist (the Kronecker product and matrix addition, subtraction, and multiplication always exist). Equation (\ref{DARE}) can be given by
\begin{equation}\label{DARE_PROOF}
P = A^\top M_1 A + Q,
\end{equation}
which is differentiable, and $M_1,M_2,M_3$ are also differentiable. Differentials are taken for (\ref{DARE_PROOF}) and each of $M_1,M_2,M_3$ as
\begin{align*}
\d \v P =&  ( \mathbf{V}_{n,n} + \mathbf{I}_{n^2}) (\mathbf{I}_n \otimes A^\top M_1) \d \v A + (A^\top \otimes A^\top ) \d \v M_1 + \d \v Q \\
\d \v M_1 =& \left[ \mathbf{I}_{n^2} - (PBM_2 B^\top \otimes \mathbf{I}_n ) - (\mathbf{I}_n \otimes PBM_2 B^\top ) \right] \d \v P  \\
& - (PB \otimes PB) \d \v M_2 - (\mathbf{I}_{n^2} + \mathbf{V}_{n,n})(PBM_2 \otimes P) \v \d B \\
\d \v M_2 =& - (M_2 \otimes M_2 )\d \v M_3 \\
\d \v M_3 =& \d \v R + (B^\top \otimes B^\top ) \d \v P + (\mathbf{I}_{m^2} + \mathbf{V}_{m,m})(\mathbf{I}_{m} \otimes B^\top P) \v \d B,
\end{align*}
then these can be combined using the differential chain rule \cite[Theorem 18.2]{magnus99} to obtain
$$
Z_1 \d \v P = Z_2 \d \v A + Z_3 \d \v B + Z_4 \d \v Q + Z_5 \d \v R.
$$
The Jacobians, as defined in Proposition \ref{proposition_DARE_diff}, therefore exist if $Z_1^{-1}$ exists. 
\end{proof}

\newpage

 \section{Verification and reduction of the prediction horizon}\label{sec:horizon_reduction}
 \begin{wrapfigure}{r}{8.cm}
\vspace{-0.8cm}
    \noindent\begin{minipage}{0.55\columnwidth}
        \begin{algorithm}[H] 
            \DontPrintSemicolon
            \KwInput{$N>0$, $\mathcal{X}_0\subseteq\mathbb{X}$, $\mathcal{M}$, $(P, K)$ from (\ref{LQR}-\ref{DARE}), $\epsilon>0$, $n_s>0$, $\eta \in(0,1)$.}
            \KwOutput{$\bar{N}, \mathcal{X}$}
            \small
            \caption{MPC horizon verification and reduction}
            \label{alg:horizon_reduction}
            $\mathcal{X}\leftarrow \mathcal{X}_0$ \;
            \While{$\mathcal{X}\supset \emptyset$}{
                $\bar{N} \leftarrow N$ \; 
                \While{$\bar{N}>0$}{
                        \text{$\mathcal{X}_{\text{sample}} \leftarrow$ \text{ $n_s$ uniform state samples, s.t.: $x\in\mathcal{X}$} }  \ %\\
                        \text{$\delta \hat{u}^\star\leftarrow$  \text{ Solution of MPC QP  (\ref{stable_system_dynamics}-\ref{QP}), $\forall x\in\mathcal{X}_{\text{sample}}$}} \ %\\
                        \If{$\|\delta \hat{u}^\star_k(x)\|\leq \epsilon, \forall k\geq{\bar{N}},\ \forall x\in \mathcal{X}_{\text{sample}}$}{
                           \textbf{return} \text{TRUE}\;
                        }
                 $\bar{N} \leftarrow \bar{N}-1$     \;  
                }
            $\mathcal{X}\leftarrow \eta \mathcal{X}$  \;  
            }
        \textbf{Procedure failed}\; %\\
        \text{$N\leftarrow N+1$}\; %\\
        \text{{Go to}  Algorithm \ref{alg:alternateDescent}}\; 
        \end{algorithm}
    \end{minipage}
    \vspace{-0.2cm}       
\end{wrapfigure}

A method is proposed for the reduction of the MPC prediction horizon after imitation learning. The idea is to be able to reproduce the infinite-horizon optimal MPC up to a tolerance $\epsilon$ with high probability. Do do so, we check that, for a candidate horizon $\bar{N}$,  the MPC action deltas, $\delta \hat{u}^\star_k$, satisfy
$\|\delta \hat{u}^\star_k\|\leq \epsilon$, for all $k\geq{\bar{N}}$. 
This means that the optimal action is equal to the LQR up to a tolerance $\epsilon$. In order to provide a high probability guarantee of this condition, we propose the use of a probabilistic verification approach, similar to \cite{bobiti_samplingdriven_nodate}. This is described in Algorithm \ref{alg:horizon_reduction}. In particular, the condition is checked on a high number, $n_s$, of initial states. These states are sampled uniformly from a set of interest $\mathcal{X}$, which can be either the state constraints $\mathbb{X}$ or an estimate of the region of attraction, $\Gamma_N$. If verified, this set is a region of attraction for the system with high probability. The relationship between the number of samples and the verification probability is discussed in \cite[Chapter 5]{bobiti_samplingdriven_nodate}. The algorithm also checks whether the infinite horizon condition has been reached for the $N$ used during training. Finally, a line search for a suitable $\mathbb{X}$ is proposed using a scaling factor $\eta\in(0,1)$. In particular, the initial set is downscaled until either an horizon is found or the set becomes empty. In latter case the search fails and the procedure returns to the training algorithm with an increased $N$. Noticeably, the proposed algorithm does not require to explicitly compute the \emph{terminal set} in which the LQR is invariant and it could be used also for non-linear MPC if an infinite-horizon (or a stabilising) terminal controller is available.

\section{Platoon Model Derivation}\label{sec_plat_model}
The problem described in Section \ref{sec_main_vehicle_platoon} can be decomposed into the regulation problem
\begin{align*}
& \begin{bmatrix}
y_2 - y_1 \\ \vdots \\ y_n - y_{n-1}
\end{bmatrix} = \begin{bmatrix}
z_2 - z_1 \\ \vdots \\ z_n - z_{n-1}
\end{bmatrix} + \textbf{1} y_{ss}, \quad \begin{bmatrix}
z_2 - z_1 \\ \vdots \\ z_n - z_{n-1}
\end{bmatrix} \to \textbf{0}, 
\end{align*}
subject to the constraints
$$
\begin{bmatrix}
z_2 - z_1 \\ \vdots \\ z_n - z_{n-1}
\end{bmatrix} \geq \textbf{1} (\underline{y} - y_{ss}), \quad \text{and} \quad \textbf{1}b \leq \ddot{y} \leq \textbf{1}a.
$$
If each vehicle is modelled as a mass then a continuous-time LTI state space model can be formed as  
\begin{equation}\label{platoon_model}
\underbrace{\begin{bmatrix}
\dot{z}_2 - \dot{z}_1 \\
\vdots \\
\dot{z}_n - \dot{z}_{n-1} \\
\ddot{z}_2 -  \ddot{z}_1 \\
\vdots \\
\ddot{z}_n - \ddot{z}_{n-1}
\end{bmatrix}}_{\dot{x}} = 
\underbrace{\begin{bmatrix}
& & & 1 \\
& & & & \ddots \\
& & & & & 1 \\
&\\
\\
\\
\end{bmatrix}}_{A_c}
\underbrace{\begin{bmatrix}
{z}_2 - {z}_1 \\
\vdots \\
{z}_n - {z}_{n-1} \\
\dot{z}_2 -  \dot{z}_1 \\
\vdots \\
\dot{z}_n - \dot{z}_{n-1}
\end{bmatrix}}_x + \underbrace{\begin{bmatrix}
\\
\\
\\
-1 & 1 \\
& \ddots & \ddots \\
& & -1 & 1
\end{bmatrix}}_{B_c}
\underbrace{\begin{bmatrix}
\ddot{y}_1 \\ \vdots \\ \ddot{y}_n
\end{bmatrix}}_u,
\end{equation}
which can then be given as
$$
\dot{x} = \begin{bmatrix}
0 & \mathbf{I} \\ 0 & 0
\end{bmatrix}x + \begin{bmatrix}
0 \\ \hat{B}
\end{bmatrix} u.
$$
If it is assumed that the control input is constant between sampling intervals $t$ and $t+ d_t$, then this can be given in discrete time as

\begin{equation}\label{discrete_time_platoon}
x_{t + d_t} = \begin{bmatrix}
\mathbf{I} & dt \mathbf{I}  \\
\mathbf{O} & \mathbf{I}
\end{bmatrix} x_t + \begin{bmatrix}
\frac{1}{2}  \hat{B}_c (dt)^2 \\ \hat{B}_c dt
\end{bmatrix} u_t, \quad \hat{B}_c = \begin{bmatrix}
-1 & 1 \\
& \ddots & \ddots \\
& & -1 & 1
\end{bmatrix}
\end{equation}

where $x_t \in \mathbb{R}^{2(n_v -1)}$, and $u \in \mathbb{R}^{n_v}$ and are subject to the constraints 
$$
x_t \geq \begin{bmatrix}
(\underline{y} - y_{ss}) \mathbf{1}_{n_v-1} \\ - \infty 
\end{bmatrix}, \quad \text{and} \quad b \textbf{1} \leq u_t \leq a \textbf{1} \quad \forall t.
$$

\section{Nonlinear models} \label{app:nonlinear}
As discussed in the main paper, our approach is currently limited to Linear Time Invariant (LTI) systems. In general, conditions for infinite-horizon optimality of systems that are not LTI are non-trivial. Some of the results on MPC stability could however be maintained, for example in the case when the LQR value function, $x^\top P x$, is a local control Lyapunov function \citep{khalil2001,rawlings_mayne_paper}. In this case, the stability and intrinsic robustness results are maintained (see \citealp{Limon2003StableCM,Limon2009}). For these system, it would be possible to use our method, for instance in combination with \cite{amos_differentiable_2018}, to provide a stable Non-linear MPC. This is however a big assumptions for systems that are very non-linear. Assessing this LQR controllability condition could be done, for instance, by training a local linear model around the target equilibrium (origin) and then checking whether the DARE is solvable. This should be performed before starting the imitation learning. We leave the study of more general systems to future work.

\end{document}